\documentclass[11pt]{article}

\usepackage[T1]{fontenc}
\usepackage[utf8]{inputenc}
\usepackage{lmodern}
\usepackage{amsmath,amssymb,amsthm,mathtools}
\usepackage{enumitem}
\usepackage{hyperref}
\usepackage[nameinlink,capitalize]{cleveref}
\usepackage{geometry}
\hypersetup{
  hidelinks,
  pdftitle={The homotopy type of the clique complex of the partition graph},
  pdfauthor={Fedor B. Lyudogovskiy}
}

\numberwithin{equation}{section}

\newtheorem{theorem}{Theorem}[section]
\newtheorem{proposition}[theorem]{Proposition}
\newtheorem{lemma}[theorem]{Lemma}
\newtheorem{corollary}[theorem]{Corollary}
\theoremstyle{definition}
\newtheorem{definition}[theorem]{Definition}
\newtheorem{example}[theorem]{Example}
\theoremstyle{remark}
\newtheorem{remark}[theorem]{Remark}

\DeclareMathOperator{\Cl}{Cl}
\newcommand{\Par}{\mathrm{Par}}

\title{The homotopy type of the clique complex of the partition graph}
\author{Fedor B. Lyudogovskiy}
\date{}

\begin{document}
\maketitle

\begin{abstract}
For each positive integer \(n\), let \(G_n\) be the graph whose vertices are the partitions of \(n\), with edges corresponding to elementary transfers of one cell between two parts, followed by reordering. Let
\[
K_n:=\Cl(G_n)
\]
be the clique complex of \(G_n\).

We prove that \(K_n\) is homotopy equivalent to a wedge of \(2\)-spheres:
\[
K_n \simeq \bigvee^{\,b_n} S^2,
\qquad
b_n=\chi(K_n)-1.
\]
Thus the homotopy type of \(K_n\) is completely determined by its Euler characteristic.

The proof has three main ingredients.

First, we classify all cliques in \(G_n\) via two canonical families of simplices, called star-simplices and top-simplices, and use them to build a canonical cover \(\mathcal C_n\) of \(K_n\).

Second, we pass to the corresponding nerve \(N_n:=N(\mathcal C_n)\), construct a second natural cover of \(N_n\), and show via the intersection poset of that cover that \(K_n\) has the homotopy type of a CW-complex of dimension at most \(2\).

Third, using an explicit height function on partitions, we prove that \(K_n\) is connected and simply connected. It follows that the reduced homology of \(K_n\) is concentrated in degree \(2\), where its rank is \(\chi(K_n)-1\), and therefore \(K_n\) has the homotopy type claimed above.

We conclude with remarks on Euler characteristics, small examples, and the integer sequences arising from these complexes.
\end{abstract}

\noindent\textbf{Keywords.}
integer partitions; partition graph; clique complex; simplicial complex; homotopy type; nerve theorem; order complex; Euler characteristic.

\medskip
\noindent\textbf{MSC 2020.}
05A17, 05C25, 55U10, 57Q70.

\section{Introduction}

Graphs on integer partitions arise naturally when partitions are viewed not as isolated combinatorial objects, but as vertices of a graph linked by elementary transformations. In this paper we consider the graph \(G_n\) whose vertices are the partitions of \(n\), and whose edges correspond to elementary transfers of one cell between two parts, followed by reordering. The associated clique complex
\[
K_n:=\Cl(G_n)
\]
records the higher-order compatibility of such moves.

The graph \(G_n\), or closely related transfer graphs on partitions, also appears in the theory of combinatorial Gray codes, where one studies Hamiltonian paths and cycles under minimal-change operations on partitions. In particular, Gray-code constructions based on transferring one unit between parts go back to Savage~\cite{Savage1989}; for restricted families of partitions see Rasmussen, Savage, and West~\cite{RasmussenSavageWest1995}, and for a modern survey see M\"utze~\cite{Mutze2023}.

Related algebro-geometric graphs on monomial ideals were studied by Altmann and Sturmfels~\cite{AltmannSturmfels2005} and by Hering and Maclagan~\cite{HeringMaclagan2012}. Their edge relation is defined via torus-orbit geometry rather than by direct combinatorial cell transfers, and in the \(T\)-graph setting it may depend on the ground field~\cite{HeringMaclagan2012}.

From the topological side, clique complexes of graphs have been studied in several other settings. Goyal, Shukla, and Singh determine the homotopy type of clique complexes of line graphs for a number of natural graph classes~\cite{GoyalShuklaSingh2022}, while Adamaszek studies the behaviour of clique complexes under graph powers and obtains explicit descriptions in important examples such as powers of cycles~\cite{Adamaszek2013}. These works provide useful context for the present paper, but they concern graph constructions different from the transfer graph on integer partitions considered here.

Closer to our partition-theoretic setting, Bal encodes ordinary partitions as binary words and studies the resulting partition graphs as subgraphs of hypercubes under Hamming-distance adjacency~\cite{Bal2022}. This is again a different graph model: in Bal's graph, edges come from changing one bit in a partition word, whereas in our graph \(G_n\) an edge corresponds to transferring one cell between two parts of a Ferrers diagram. To the best of our knowledge, the homotopy type of the clique complex of this transfer graph has not been determined before.

Our main result shows that, despite the rich local combinatorics of \(K_n\), its global homotopy type admits a particularly simple description.

\begin{theorem}\label{thm:intro-main}
For every \(n\ge 1\), the clique complex \(K_n\) is homotopy equivalent to a wedge of \(2\)-spheres:
\[
K_n \simeq \bigvee^{\,b_n} S^2,
\qquad
b_n=\chi(K_n)-1.
\]
\end{theorem}

Equivalently, the homotopy type of \(K_n\) is completely determined by its Euler characteristic.

The proof proceeds in three stages.

First, we analyze cliques in the partition graph \(G_n\). We show that all triangles, and hence all cliques, are governed by two canonical local mechanisms. These give rise to two distinguished families of simplices in \(K_n\), called \emph{star-simplices} and \emph{top-simplices}. We prove that every clique in \(G_n\) is contained in one of these two families, and from this obtain a complete classification of the maximal simplices of \(K_n\).

Second, we use these canonical simplices to construct a cover \(\mathcal C_n\) of \(K_n\) by full star- and full top-simplices. This cover is a good cover, so its nerve
\[
N_n:=N(\mathcal C_n)
\]
is homotopy equivalent to \(K_n\). We then study a second natural cover of the nerve \(N_n\), indexed by the vertices of \(K_n\). The associated intersection poset has order complex of dimension at most \(2\), and this yields a \(2\)-dimensional model for both \(N_n\) and \(K_n\).

Third, we prove that \(K_n\) is connected and simply connected. The connectedness is established directly. For simple connectedness, we introduce an explicit height function on partitions and study edge-loops in the \(1\)-skeleton of \(K_n\), which is the graph \(G_n\). A local peak-reduction lemma shows that any loop can be simplified by replacing a maximal-height vertex by a homotopic path with strictly lower intermediate height. Since adjacent vertices always have different heights, maximal-height vertices are automatically isolated peaks, and an induction on loop complexity yields \(\pi_1(K_n)=0\).

At this point the proof of \Cref{thm:intro-main} follows from standard facts about simply connected CW-complexes of dimension at most \(2\). Indeed, \(K_n\) has the homotopy type of a simply connected CW-complex of dimension at most \(2\). Hence its reduced homology is concentrated in degree \(2\), where the rank is determined by the reduced Euler characteristic:
\[
\operatorname{rank}\widetilde H_2(K_n)=\widetilde\chi(K_n)=\chi(K_n)-1.
\]
Therefore \(K_n\) is homotopy equivalent to a wedge of exactly \(\chi(K_n)-1\) copies of \(S^2\).

The paper is organized as follows. In Section~\ref{sec:partitions-moves-canonical-simplices} we introduce the partition graph, the clique complex, the canonical simplices, and the height function. In Section~\ref{sec:classification-of-cliques} we classify triangles, cliques, and maximal simplices. In Section~\ref{sec:good-cover-and-nerve} we construct a canonical good cover of \(K_n\) by full star- and full top-simplices and pass to its nerve. In Section~\ref{sec:anchor-cover-2-model} we prove that \(N_n\), and hence also \(K_n\), has the homotopy type of a CW-complex of dimension at most \(2\). In Section~\ref{sec:connectedness-and-simple-connectedness} we prove connectedness and simple connectedness. Section~\ref{sec:homotopy-type} contains the proof of the main theorem. In the final section we discuss Euler characteristics, small examples, and the resulting integer sequences.

\medskip

The qualitative conclusion is the following contrast: although the local simplex structure of \(K_n\) becomes arbitrarily complicated as \(n\) grows, its global homotopy type always remains that of a wedge of \(2\)-spheres.

\section{Partitions, elementary moves, and canonical simplices}
\label{sec:partitions-moves-canonical-simplices}

In this section we fix notation for partitions and Ferrers diagrams, define the partition graph \(G_n\), and introduce the two canonical families of simplices that govern the clique structure of the clique complex
\[
K_n:=\Cl(G_n).
\]

\subsection{Partitions and Ferrers diagrams}

A partition of a positive integer \(n\) is a finite weakly decreasing sequence
\[
\lambda=(\lambda_1,\lambda_2,\dots,\lambda_\ell),
\qquad
\lambda_1\ge \lambda_2\ge \cdots \ge \lambda_\ell>0,
\qquad
\sum_{i=1}^{\ell}\lambda_i=n.
\]
We write \(\lambda\vdash n\) to indicate that \(\lambda\) is a partition of \(n\), and denote by \(\Par(n)\) the set of all partitions of \(n\).

We identify \(\lambda\) with its Ferrers diagram. Thus \(\lambda\) may be viewed either as a sequence of row lengths or as a left-justified arrangement of cells with \(\lambda_i\) cells in the \(i\)-th row. We freely pass between these two descriptions. For standard background on integer partitions and Ferrers diagrams, see, for example, Andrews~\cite{Andrews1976}.

A \emph{removable corner} of \(\lambda\) is a cell whose removal again yields a Ferrers diagram of a partition. Equivalently, it is the last cell in a row whose length is strictly greater than the length of the next row, where by convention \(\lambda_{\ell+1}=0\).

An \emph{addable corner} of \(\lambda\) is a position in which a cell can be added so that the resulting diagram is again a Ferrers diagram of a partition.

If \(c\) is a removable corner or \(a\) an addable corner, we write \(r(c)\) and \(r(a)\) for their row numbers, and \(\operatorname{col}(c)\) and \(\operatorname{col}(a)\) for their column numbers. If \(\lambda'\) denotes the conjugate partition, then
\[
r(c)=\lambda'_{\operatorname{col}(c)}
\]
for every removable corner \(c\), while
\[
r(a)=\lambda'_{\operatorname{col}(a)}+1
\]
for every addable corner \(a\).
In particular, a removable corner is uniquely determined by its column index, and an addable corner is uniquely determined by its column index, since the displayed formulas recover the corresponding row number.

For brevity, we sometimes write
\[
k(c):=\operatorname{col}(c),
\qquad
\ell(a):=\operatorname{col}(a)
\]
for the column indices of removable and addable corners. This notation should not be confused with the index \(\ell\) in
\[
\lambda=(\lambda_1,\dots,\lambda_\ell),
\]
which denotes the number of nonzero parts of \(\lambda\).

\subsection{The partition graph}

Fix \(n\ge 1\). The \emph{partition graph} \(G_n\) is the graph whose vertex set is \(\Par(n)\). Two partitions \(\lambda,\mu\in\Par(n)\) are adjacent if one can be obtained from the other by removing one cell from a row of the Ferrers diagram, adding it to another row, and then reordering the row lengths into weakly decreasing order.

More explicitly, let \(c\) be a removable corner of \(\lambda\) and \(a\) an addable corner of \(\lambda\), with \(r(c)\neq r(a)\). If removing the cell \(c\) from the Ferrers diagram of \(\lambda\), adding a cell at \(a\), and then reordering the row lengths yields a partition of \(n\) different from \(\lambda\), we call the transfer \emph{admissible} and write
\[
\lambda(c\to a)
\]
for the resulting partition. Every neighbor of \(\lambda\) in \(G_n\) is obtained from \(\lambda\) by such an admissible transfer.

On the level of conjugate partitions, an admissible transfer \(\lambda(c\to a)\) changes exactly two parts: it decreases the part indexed by \(\operatorname{col}(c)\) by \(1\) and increases the part indexed by \(\operatorname{col}(a)\) by \(1\). Equivalently, if \(e_j\) denotes the \(j\)-th standard basis vector, then
\[
(\lambda(c\to a))' = \lambda' - e_{\operatorname{col}(c)} + e_{\operatorname{col}(a)}.
\]
In particular, two partitions are adjacent in \(G_n\) if and only if their conjugates differ by decreasing one part by \(1\) and increasing another part by \(1\).

\begin{lemma}\label{lem:adjacency-conjugate-criterion}
Let \(\lambda,\mu\in\Par(n)\). Then \(\lambda\) and \(\mu\) are adjacent in \(G_n\) if and only if there exist a removable corner \(c\) of \(\lambda\) and an addable corner \(a\) of \(\lambda\) such that the transfer \(\lambda(c\to a)\) is admissible and
\[
\mu'=\lambda'-e_{\operatorname{col}(c)}+e_{\operatorname{col}(a)}.
\]
\end{lemma}

\begin{proof}
This is exactly the conjugate-coordinate description of an admissible transfer recorded above.
\end{proof}

We write
\[
K_n:=\Cl(G_n)
\]
for the clique complex of \(G_n\). Thus the simplices of \(K_n\) are precisely the finite cliques in \(G_n\).

\begin{lemma}\label{lem:basic-move-properties}
Let \(\lambda,\mu\in\Par(n)\).
\begin{enumerate}
    \item If \(\lambda\) is adjacent to \(\mu\) in \(G_n\), then \(\mu\) is adjacent to \(\lambda\).
    \item If \(\mu=\lambda(c\to a)\), then \(\lambda\) and \(\mu\) differ by exactly one elementary transfer of one cell.
    \item Every edge of \(G_n\) preserves the total number of cells and hence stays inside \(\Par(n)\).
\end{enumerate}
\end{lemma}

\begin{proof}
All three statements are immediate from the definition.
\end{proof}

\subsection{Canonical simplices}

We now introduce the two canonical families of simplices that govern the clique structure of \(K_n\).

\begin{definition}\label{def:star-simplex}
Let \(\lambda\vdash n\), and let \(c\) be a removable corner of \(\lambda\). Let \(A\) be a finite set of addable corners such that every transfer \(\lambda(c\to a)\) is admissible. The simplex
\[
\{\lambda\}\cup\{\lambda(c\to a):a\in A\}
\]
is called a \emph{star-simplex} based at \((\lambda,c)\).
\end{definition}

\begin{definition}\label{def:top-simplex}
Let \(\lambda\vdash n\), and let \(a\) be an addable corner of \(\lambda\). Let \(C\) be a finite set of removable corners such that every transfer \(\lambda(c\to a)\) is admissible. The simplex
\[
\{\lambda\}\cup\{\lambda(c\to a):c\in C\}
\]
is called a \emph{top-simplex} based at \((\lambda,a)\).
\end{definition}

These definitions reflect the two basic local mechanisms that generate triangles in \(G_n\): in the star case the removable corner is fixed and the addable corner varies, whereas in the top case the addable corner is fixed and the removable corner varies. In Section~\ref{sec:classification-of-cliques} we prove that these two mechanisms account for all triangles and, more generally, for all cliques in \(G_n\); see \Cref{thm:triangle-classification,thm:clique-classification}.

\subsection{Full star- and top-simplices}

For later use we also record the maximal simplices inside the two canonical families.

\begin{definition}\label{def:full-star-top}
Let \(\lambda\vdash n\).
\begin{enumerate}
    \item For a removable corner \(c\) of \(\lambda\), define
    \[
    A_{\max}(\lambda,c)
    :=
    \{\,a : \lambda(c\to a)\text{ is admissible}\,\}.
    \]
    The corresponding \emph{full star-simplex} is
    \[
    \Sigma^{\mathrm{star}}_{\max}(\lambda,c)
    :=
    \{\lambda\}\cup\{\lambda(c\to a):a\in A_{\max}(\lambda,c)\}.
    \]

    \item For an addable corner \(a\) of \(\lambda\), define
    \[
    C_{\max}(\lambda,a)
    :=
    \{\,c : \lambda(c\to a)\text{ is admissible}\,\}.
    \]
    The corresponding \emph{full top-simplex} is
    \[
    \Sigma^{\mathrm{top}}_{\max}(\lambda,a)
    :=
    \{\lambda\}\cup\{\lambda(c\to a):c\in C_{\max}(\lambda,a)\}.
    \]
\end{enumerate}
\end{definition}

By construction, every star-simplex is contained in a full star-simplex, and every top-simplex is contained in a full top-simplex.

\subsection{The height function}

For the proof of simple connectedness in Section~\ref{sec:connectedness-and-simple-connectedness} we shall use the following height function on partitions:
\[
h(\lambda):=\sum_{i=1}^{\ell} i\,\lambda_i.
\]

Equivalently, if \(\lambda'\) denotes the conjugate partition, then
\[
h(\lambda)=\sum_{j\ge 1}\frac{\lambda_j'(\lambda_j'+1)}{2}.
\]

\begin{lemma}\label{lem:height-change}
If
\[
\mu=\lambda(c\to a),
\]
then
\[
h(\mu)-h(\lambda)=r(a)-r(c).
\]
In particular, if
\[
r(a)<r(c),
\]
then
\[
h(\mu)<h(\lambda).
\]
\end{lemma}

\begin{proof}
Using the conjugate partition, we may write
\[
\mu' = \lambda' - e_{\operatorname{col}(c)} + e_{\operatorname{col}(a)}.
\]
Since
\[
h(\lambda)=\sum_{j\ge 1}\frac{\lambda'_j(\lambda'_j+1)}{2},
\]
only the two columns \(\operatorname{col}(c)\) and \(\operatorname{col}(a)\) contribute to the change in height. Therefore
\begin{align*}
 h(\mu)-h(\lambda)
 &=\left(\frac{(\lambda'_{\operatorname{col}(a)}+1)(\lambda'_{\operatorname{col}(a)}+2)}{2}-\frac{\lambda'_{\operatorname{col}(a)}(\lambda'_{\operatorname{col}(a)}+1)}{2}\right)\\
 &\qquad-\left(\frac{\lambda'_{\operatorname{col}(c)}(\lambda'_{\operatorname{col}(c)}+1)}{2}-\frac{(\lambda'_{\operatorname{col}(c)}-1)\lambda'_{\operatorname{col}(c)}}{2}\right)\\
 &=\lambda'_{\operatorname{col}(a)}+1-\lambda'_{\operatorname{col}(c)}\\
 &=r(a)-r(c).
\end{align*}
Here the last equality uses the identities
\[
r(a)=\lambda'_{\operatorname{col}(a)}+1,\qquad r(c)=\lambda'_{\operatorname{col}(c)}.
\]
\end{proof}

\begin{lemma}\label{lem:adjacent-different-heights}
If \(\lambda\) and \(\mu\) are distinct adjacent vertices of \(G_n\), then
\[
h(\lambda)\neq h(\mu).
\]
\end{lemma}

\begin{proof}
Write \(\mu=\lambda(c\to a)\). Since the transfer is nontrivial, one has \(r(a)\neq r(c)\). Now apply Lemma~\ref{lem:height-change}.
\end{proof}

\subsection{Summary}

We have introduced the partition graph \(G_n\), its clique complex \(K_n\), the canonical star- and top-simplices, and the height function used later in the proof of simple connectedness. In the next section we prove that every clique in \(G_n\) is contained in one of the two canonical simplex families.

\section{Classification of cliques in the partition graph}
\label{sec:classification-of-cliques}

In this section we prove that every clique in the partition graph \(G_n\) is contained in one of the two canonical simplex families introduced in Section~2. The key step is a complete classification of triangles in \(G_n\).

\subsection{Triangles in the partition graph}

We begin with the two basic triangle constructions already implicit in the definitions of star- and top-simplices.

\begin{lemma}\label{lem:star-triangle-criterion}
Let
\[
\mu_1=\lambda(c\to a_1),\qquad
\mu_2=\lambda(c\to a_2).
\]
Then the vertices \(\lambda,\mu_1,\mu_2\) span a triangle in \(G_n\).
\end{lemma}

\begin{proof}
Let \(k=\operatorname{col}(c)\), \(\ell_1=\operatorname{col}(a_1)\), and \(\ell_2=\operatorname{col}(a_2)\). On the level of conjugate partitions,
\[
\mu_1' = \lambda' - e_k + e_{\ell_1},\qquad \mu_2' = \lambda' - e_k + e_{\ell_2}.
\]
Hence
\[
\mu_2' = \mu_1' - e_{\ell_1} + e_{\ell_2}.
\]
Therefore \(\mu_1'\) and \(\mu_2'\) differ by decreasing one part by \(1\) and increasing another part by \(1\). By Lemma~\ref{lem:adjacency-conjugate-criterion}, this means that \(\mu_1\) and \(\mu_2\) are adjacent in \(G_n\). Since both are also adjacent to \(\lambda\), the three vertices span a triangle in \(G_n\).
\end{proof}

\begin{lemma}\label{lem:top-triangle-criterion}
Let
\[
\mu_1=\lambda(c_1\to a),\qquad
\mu_2=\lambda(c_2\to a).
\]
Then the vertices \(\lambda,\mu_1,\mu_2\) span a triangle in \(G_n\).
\end{lemma}

\begin{proof}
Let \(k_1=\operatorname{col}(c_1)\), \(k_2=\operatorname{col}(c_2)\), and \(\ell=\operatorname{col}(a)\). On the level of conjugate partitions,
\[
\mu_1' = \lambda' - e_{k_1} + e_{\ell},\qquad \mu_2' = \lambda' - e_{k_2} + e_{\ell}.
\]
Hence
\[
\mu_2' = \mu_1' - e_{k_2} + e_{k_1}.
\]
Therefore \(\mu_1'\) and \(\mu_2'\) differ by decreasing one part by \(1\) and increasing another part by \(1\). By Lemma~\ref{lem:adjacency-conjugate-criterion}, this means that \(\mu_1\) and \(\mu_2\) are adjacent in \(G_n\). Since both are also adjacent to \(\lambda\), the three vertices span a triangle in \(G_n\).
\end{proof}

The converse statement is the main local combinatorial fact.

\begin{lemma}\label{lem:forbidden-mixed-triangle}
Let
\[
\mu_1=\lambda(c_1\to a_1),\qquad
\mu_2=\lambda(c_2\to a_2),
\]
and assume that
\[
c_1\neq c_2,\qquad a_1\neq a_2.
\]
Then \(\mu_1\) and \(\mu_2\) are not adjacent in \(G_n\).
\end{lemma}

\begin{proof}
Let
\[
k_1=\operatorname{col}(c_1),\qquad k_2=\operatorname{col}(c_2),\qquad \ell_1=\operatorname{col}(a_1),\qquad \ell_2=\operatorname{col}(a_2).
\]
Since distinct removable corners have distinct column indices, and distinct addable corners also have distinct column indices, the assumptions \(c_1\neq c_2\) and \(a_1\neq a_2\) imply
\[
k_1\neq k_2,
\qquad
\ell_1\neq \ell_2.
\]
On the level of conjugate partitions,
\[
\mu_1' = \lambda' - e_{k_1} + e_{\ell_1},\qquad \mu_2' = \lambda' - e_{k_2} + e_{\ell_2}.
\]
Hence
\[
\mu_2'-\mu_1' = e_{k_1}-e_{\ell_1}-e_{k_2}+e_{\ell_2}.
\]
If \(\mu_1\) and \(\mu_2\) were adjacent in \(G_n\), then by Lemma~\ref{lem:adjacency-conjugate-criterion} their conjugates would differ by a vector of the form \(-e_u+e_v\) with \(u\neq v\), that is, by exactly one decrement and one increment.

This is impossible for the vector above. If the four indices \(k_1,k_2,\ell_1,\ell_2\) are pairwise distinct, then the support has size \(4\). If there is a cross-equality such as \(k_1=\ell_2\) or \(k_2=\ell_1\), then one obtains a coefficient of magnitude \(2\) or a vector with support of size \(3\). In no case can \(\mu_2'-\mu_1'\) be of the form \(-e_u+e_v\). Therefore \(\mu_1\) and \(\mu_2\) are not adjacent in \(G_n\).
\end{proof}

Combining the three lemmas above, we obtain a complete classification of triangles.

\begin{theorem}\label{thm:triangle-classification}
Let
\[
\mu_1=\lambda(c_1\to a_1),\qquad
\mu_2=\lambda(c_2\to a_2).
\]
Then the vertices \(\lambda,\mu_1,\mu_2\) form a triangle in \(G_n\) if and only if
\[
c_1=c_2
\qquad\text{or}\qquad
a_1=a_2.
\]
\end{theorem}

\begin{proof}
If \(c_1=c_2\), the claim follows from Lemma~\ref{lem:star-triangle-criterion}. If \(a_1=a_2\), it follows from Lemma~\ref{lem:top-triangle-criterion}. Conversely, if \(c_1\neq c_2\) and \(a_1\neq a_2\), then Lemma~\ref{lem:forbidden-mixed-triangle} shows that \(\mu_1\) and \(\mu_2\) are not adjacent, so no triangle exists.
\end{proof}

\subsection{Uniformity of cliques through a fixed vertex}

The triangle classification immediately forces a strong uniformity property for cliques containing a fixed vertex.

\begin{lemma}\label{lem:uniform-type}
Let
\[
\mu_i=\lambda(c_i\to a_i)\qquad (i=1,\dots,m),
\]
and suppose that
\[
\{\lambda,\mu_1,\dots,\mu_m\}
\]
is a clique in \(G_n\). Then either
\[
c_1=\cdots=c_m,
\]
or
\[
a_1=\cdots=a_m.
\]
\end{lemma}

\begin{proof}
For every pair \(i\neq j\), the vertices \(\lambda,\mu_i,\mu_j\) form a triangle. By Theorem~\ref{thm:triangle-classification}, for each such pair one has
\[
c_i=c_j
\qquad\text{or}\qquad
a_i=a_j.
\]

Suppose first that there exist indices \(i,j\) with \(c_i\neq c_j\). Then necessarily \(a_i=a_j\). Let \(k\) be arbitrary. Applying Theorem~\ref{thm:triangle-classification} to the pairs \((i,k)\) and \((j,k)\), we obtain
\[
c_i=c_k \text{ or } a_i=a_k,
\qquad
c_j=c_k \text{ or } a_j=a_k.
\]
Since \(c_i\neq c_j\) and \(a_i=a_j\), the only possibility is \(a_k=a_i\). As \(k\) was arbitrary, all \(a_k\) coincide.

The dual argument shows that if there exist \(i,j\) with \(a_i\neq a_j\), then all \(c_i\) coincide.
\end{proof}

We may now classify all cliques in \(G_n\).

\begin{theorem}\label{thm:clique-classification}
Every clique in \(G_n\) is contained in a star-simplex or in a top-simplex.
\end{theorem}

\begin{proof}
Let \(C\) be a clique in \(G_n\), and choose a vertex \(\lambda\in C\). Every other vertex of \(C\) is a neighbor of \(\lambda\), hence is of the form
\[
\lambda(c_i\to a_i).
\]
By Lemma~\ref{lem:uniform-type}, either all \(c_i\) are equal or all \(a_i\) are equal. In the first case \(C\) is contained in a star-simplex based at \((\lambda,c)\); in the second case it is contained in a top-simplex based at \((\lambda,a)\).
\end{proof}

\subsection{Full star- and top-simplices}

We now identify the maximal simplices inside the star- and top-families.

\begin{proposition}\label{prop:maximal-in-star-family}
A star-simplex is maximal among star-simplices if and only if it is a full star-simplex.
\end{proposition}

\begin{proof}
Let
\[
\Sigma=\{\lambda\}\cup\{\lambda(c\to a):a\in A\}
\]
be a star-simplex based at \((\lambda,c)\). If \(A\neq A_{\max}(\lambda,c)\), then there exists an admissible addable corner \(a\notin A\), and adjoining the vertex \(\lambda(c\to a)\) yields a strictly larger star-simplex. Thus \(\Sigma\) is not maximal in the star-family.

Conversely, if \(A=A_{\max}(\lambda,c)\), then by definition no further admissible vertex of the form \(\lambda(c\to a)\) can be added while keeping the removable corner \(c\) fixed. Hence \(\Sigma\) is maximal among star-simplices.
\end{proof}

\begin{proposition}\label{prop:maximal-in-top-family}
A top-simplex is maximal among top-simplices if and only if it is a full top-simplex.
\end{proposition}

\begin{proof}
This is dual to Proposition~\ref{prop:maximal-in-star-family}.
\end{proof}

\subsection{Genuine maximality in \texorpdfstring{\(K_n\)}{K\_n}}

We next determine which full star- and top-simplices are maximal simplices of the whole clique complex \(K_n\).

\begin{proposition}\label{prop:large-full-star-maximal}
Let
\[
\Sigma=\Sigma^{\mathrm{star}}_{\max}(\lambda,c),
\]
and assume that
\[
|A_{\max}(\lambda,c)|\ge 2.
\]
Then \(\Sigma\) is a maximal simplex of \(K_n\).
\end{proposition}

\begin{proof}
Let \(\nu\) be a vertex adjacent to every vertex of \(\Sigma\). Since \(\nu\sim \lambda\), we may write
\[
\nu=\lambda(d\to b)
\]
for some removable corner \(d\) and addable corner \(b\).

Choose two distinct corners \(a_1,a_2\in A_{\max}(\lambda,c)\), and let
\[
v_i=\lambda(c\to a_i)\qquad (i=1,2).
\]
Since \(\nu\) is adjacent to both \(v_1\) and \(v_2\), the triples \(\lambda,v_1,\nu\) and \(\lambda,v_2,\nu\) are triangles. By Theorem~\ref{thm:triangle-classification}, we obtain
\[
d=c \text{ or } b=a_1,
\qquad
d=c \text{ or } b=a_2.
\]
As \(a_1\neq a_2\), it follows that \(d=c\). Hence
\[
\nu=\lambda(c\to b).
\]
Since this move is admissible, one has \(b\in A_{\max}(\lambda,c)\), so \(\nu\in\Sigma\). Therefore no new vertex can be added to \(\Sigma\), and \(\Sigma\) is maximal in \(K_n\).
\end{proof}

\begin{proposition}\label{prop:large-full-top-maximal}
Let
\[
\Sigma=\Sigma^{\mathrm{top}}_{\max}(\lambda,a),
\]
and assume that
\[
|C_{\max}(\lambda,a)|\ge 2.
\]
Then \(\Sigma\) is a maximal simplex of \(K_n\).
\end{proposition}

\begin{proof}
This is dual to Proposition~\ref{prop:large-full-star-maximal}.
\end{proof}

The one-dimensional case requires a separate treatment.

\begin{proposition}\label{prop:maximal-edge-classification}
Let
\[
e=\{\lambda,\lambda(c\to a)\}
\]
be an edge of \(G_n\). Then \(e\) is a maximal simplex of \(K_n\) if and only if
\[
|A_{\max}(\lambda,c)|=1
\qquad\text{and}\qquad
|C_{\max}(\lambda,a)|=1.
\]
\end{proposition}

\begin{proof}
Suppose first that \(|A_{\max}(\lambda,c)|\ge 2\). Choose \(a'\neq a\) in \(A_{\max}(\lambda,c)\). Then
\[
\{\lambda,\lambda(c\to a),\lambda(c\to a')\}
\]
is a star-triangle containing \(e\), so \(e\) is not maximal. Dually, if \(|C_{\max}(\lambda,a)|\ge 2\), then \(e\) lies in a top-triangle and is again not maximal.

Conversely, assume
\[
A_{\max}(\lambda,c)=\{a\},
\qquad
C_{\max}(\lambda,a)=\{c\}.
\]
Let \(\nu\) be a vertex adjacent to both endpoints of \(e\). Since \(\nu\sim\lambda\), write
\[
\nu=\lambda(d\to b).
\]
Then \(\lambda,\lambda(c\to a),\nu\) form a triangle, so by Theorem~\ref{thm:triangle-classification} either \(d=c\) or \(b=a\).

If \(d=c\), then
\[
\nu=\lambda(c\to b),
\]
and since \(A_{\max}(\lambda,c)=\{a\}\), one must have \(b=a\), so \(\nu=\lambda(c\to a)\). If \(b=a\), then
\[
\nu=\lambda(d\to a),
\]
and since \(C_{\max}(\lambda,a)=\{c\}\), one must have \(d=c\), so again \(\nu=\lambda(c\to a)\). Thus no third vertex distinct from the endpoints can be adjacent to both of them. Hence \(e\) is maximal.
\end{proof}

\begin{corollary}\label{cor:maximal-simplices-classification}
The maximal simplices of \(K_n\) are exactly:
\begin{enumerate}
    \item all full star-simplices \(\Sigma^{\mathrm{star}}_{\max}(\lambda,c)\) with \(|A_{\max}(\lambda,c)|\ge 2\);
    \item all full top-simplices \(\Sigma^{\mathrm{top}}_{\max}(\lambda,a)\) with \(|C_{\max}(\lambda,a)|\ge 2\);
    \item all edges
    \[
    \{\lambda,\lambda(c\to a)\}
    \]
    such that
    \[
    |A_{\max}(\lambda,c)|=1
    \qquad\text{and}\qquad
    |C_{\max}(\lambda,a)|=1.
    \]
\end{enumerate}
\end{corollary}

\begin{proof}
This follows from Theorem~\ref{thm:clique-classification}, Propositions~\ref{prop:maximal-in-star-family} and \ref{prop:maximal-in-top-family}, Propositions~\ref{prop:large-full-star-maximal} and \ref{prop:large-full-top-maximal}, and Proposition~\ref{prop:maximal-edge-classification}.
\end{proof}

\begin{lemma}\label{lem:every-simplex-in-maximal}
Every simplex of \(K_n\) is contained in a maximal simplex of \(K_n\).
\end{lemma}

\begin{proof}
Let \(\sigma\) be a simplex of \(K_n\). By Theorem~\ref{thm:clique-classification}, \(\sigma\) is contained in a star-simplex or a top-simplex. Enlarging within that family yields a full star-simplex or a full top-simplex.

If the resulting full simplex is maximal in \(K_n\), we are done. Otherwise it is one-dimensional, hence an edge, and Proposition~\ref{prop:maximal-edge-classification} shows that it lies in a larger simplex of the opposite type. In either case, \(\sigma\) is contained in a maximal simplex of \(K_n\).
\end{proof}

\begin{corollary}\label{cor:canonical-cover}
The maximal simplices of \(K_n\) cover \(K_n\).
\end{corollary}

\begin{proof}
Immediate from Lemma~\ref{lem:every-simplex-in-maximal}.
\end{proof}

\subsection{Summary}

We have shown that every clique in \(G_n\) is controlled by one of the two canonical mechanisms encoded by star- and top-simplices. In particular, the maximal simplices of \(K_n\) are completely classified, and they form a canonical cover of \(K_n\). In the next section we use this cover to pass to the nerve of \(K_n\).

\section{A canonical star/top cover and its nerve}
\label{sec:good-cover-and-nerve}

In this section we replace the cover of \(K_n\) by maximal simplices with a more uniform canonical cover by full star- and full top-simplices. This cover is better suited to the analysis carried out in the next section.

\subsection{The canonical cover}

Recall from Definition~\ref{def:full-star-top} that for every partition \(\lambda\vdash n\), every removable corner \(c\) of \(\lambda\), and every addable corner \(a\) of \(\lambda\), we have the full star-simplex
\[
\Sigma^{\mathrm{star}}_{\max}(\lambda,c)
=
\{\lambda\}\cup\{\lambda(c\to b): b\in A_{\max}(\lambda,c)\}
\]
and the full top-simplex
\[
\Sigma^{\mathrm{top}}_{\max}(\lambda,a)
=
\{\lambda\}\cup\{\lambda(d\to a): d\in C_{\max}(\lambda,a)\}.
\]

Let
\[
\mathcal C_n
\]
denote the family consisting of all distinct full star- and full top-simplices.

\begin{proposition}\label{prop:canonical-star-top-cover}
The family \(\mathcal C_n\) covers \(K_n\):
\[
K_n=\bigcup_{U\in\mathcal C_n} U.
\]
\end{proposition}

\begin{proof}
Let \(\sigma\) be a simplex of \(K_n\). By Theorem~\ref{thm:clique-classification}, \(\sigma\) is contained in a star-simplex or in a top-simplex. By Definition~\ref{def:full-star-top}, every star-simplex is contained in a full star-simplex and every top-simplex is contained in a full top-simplex. Hence \(\sigma\subseteq U\) for some \(U\in\mathcal C_n\). Since \(\sigma\) was arbitrary, the family \(\mathcal C_n\) covers \(K_n\).
\end{proof}

\subsection{Intersections of members of the cover}

The key observation is completely general.

\begin{lemma}\label{lem:intersection-of-simplices}
Let
\[
U_1,\dots,U_r\in \mathcal C_n.
\]
Then
\[
U_1\cap\cdots\cap U_r
\]
is either empty or a simplex of \(K_n\).
\end{lemma}

\begin{proof}
Each \(U_i\in\mathcal C_n\) is a simplex of \(K_n\). The intersection of finitely many simplices in a simplicial complex is either empty or a simplex.
\end{proof}

\begin{corollary}\label{cor:cover-intersections-contractible}
Every nonempty finite intersection of members of \(\mathcal C_n\) is contractible.
\end{corollary}

\begin{proof}
By Lemma~\ref{lem:intersection-of-simplices}, every nonempty finite intersection of members of \(\mathcal C_n\) is a simplex of \(K_n\), hence contractible.
\end{proof}

\subsection{The nerve of the canonical cover}

Let
\[
N_n:=N(\mathcal C_n)
\]
be the nerve of the cover \(\mathcal C_n\). Thus the vertices of \(N_n\) are the full star- and full top-simplices in \(\mathcal C_n\), and a finite subfamily
\[
\{U_0,\dots,U_p\}\subseteq \mathcal C_n
\]
spans a \(p\)-simplex of \(N_n\) if and only if
\[
U_0\cap\cdots\cap U_p\neq\varnothing.
\]

\begin{theorem}\label{thm:canonical-cover-good}
The family \(\mathcal C_n\) is a good cover of \(K_n\) in the simplicial sense: each member of \(\mathcal C_n\) is contractible, and every nonempty finite intersection of members of \(\mathcal C_n\) is contractible.
\end{theorem}

\begin{proof}
By Proposition~\ref{prop:canonical-star-top-cover}, the family \(\mathcal C_n\) covers \(K_n\). Each \(U\in\mathcal C_n\) is a simplex, hence contractible. By Corollary~\ref{cor:cover-intersections-contractible}, every nonempty finite intersection of members of \(\mathcal C_n\) is contractible. This is exactly the required good-cover property.
\end{proof}

\begin{corollary}\label{cor:Kn-nerve-equivalence}\label{cor:nerve-equivalence}
The clique complex \(K_n\) is homotopy equivalent to the nerve \(N_n\):
\[
K_n\simeq N_n.
\]
\end{corollary}

\begin{proof}
Apply the nerve theorem for finite covers by contractible subcomplexes with contractible nonempty finite intersections; see, for instance, \cite[Section~10.6]{Kozlov2008}. The hypotheses are verified by Theorem~\ref{thm:canonical-cover-good}.
\end{proof}

\subsection{Summary}

We have constructed a canonical cover \(\mathcal C_n\) of \(K_n\) by full star- and full top-simplices and proved that it is a good cover. Hence
\[
K_n\simeq N_n,
\qquad
N_n:=N(\mathcal C_n).
\]
In the next section we do not collapse \(K_n\) directly. Instead, we work on the nerve \(N_n\) itself, construct a second natural cover of \(N_n\), and use its intersection poset to obtain a \(2\)-dimensional model.

\section{The anchor cover and a \texorpdfstring{\(2\)}{2}-dimensional model}
\label{sec:anchor-cover-2-model}\label{sec:collapsing-Kn}

In this section we replace the earlier direct collapse approach by a different reduction scheme.
We work not directly with \(K_n\), but with the nerve
\[
N_n:=N(\mathcal C_n)
\]
of the canonical cover \(\mathcal C_n\) introduced in Section~\ref{sec:good-cover-and-nerve}. By Corollary~\ref{cor:Kn-nerve-equivalence}, we already know that
\[
K_n\simeq N_n,
\]
so it suffices to construct a \(2\)-dimensional model for \(N_n\).
We first construct a second natural cover of \(N_n\), indexed by the vertices of \(K_n\), and then study the poset of its nonempty intersections. The crucial point is that this intersection poset has dimension at most \(2\).

\subsection{The anchor cover of \texorpdfstring{\(N_n\)}{N\_n}}

For every partition \(\lambda\vdash n\), define
\[
A_\lambda:=\{\,U\in\mathcal C_n:\lambda\in U\,\}.
\]
Since every member of \(A_\lambda\) contains \(\lambda\), the whole set \(A_\lambda\) spans a simplex of the nerve \(N_n\); we denote this simplex again by \(A_\lambda\subseteq N_n\).

\begin{lemma}\label{lem:anchor-cover}
The family
\[
\mathcal A_n:=\{A_\lambda:\lambda\vdash n\}
\]
covers \(N_n\).
\end{lemma}

\begin{proof}
Let \(\sigma\) be a simplex of \(N_n\). By definition of the nerve, its common intersection
\[
I(\sigma):=\bigcap_{U\in\sigma}U
\]
is nonempty. Choose any \(\lambda\in I(\sigma)\). Then \(\lambda\in U\) for every vertex \(U\in\sigma\), so every \(U\in\sigma\) belongs to \(V_\lambda\). Hence \(\sigma\subseteq A_\lambda\). Since \(\sigma\) was arbitrary, the family \(\mathcal A_n\) covers \(N_n\).
\end{proof}

\begin{proposition}\label{prop:anchor-intersections-basic}
Let
\[
S=\{\lambda_0,\dots,\lambda_p\}
\]
be a finite set of partitions. Then
\[
A_{\lambda_0}\cap\cdots\cap A_{\lambda_p}\neq\varnothing
\]
if and only if \(S\) is a simplex of \(K_n\). Whenever the intersection is nonempty, it is itself a simplex of \(N_n\).
\end{proposition}

\begin{proof}
Suppose first that
\[
A_{\lambda_0}\cap\cdots\cap A_{\lambda_p}\neq\varnothing.
\]
Then there exists some \(U\in\mathcal C_n\) containing all \(\lambda_0,\dots,\lambda_p\). Since \(U\) is a simplex of \(K_n\), the set \(S\) is a simplex of \(K_n\).

Conversely, suppose \(S\) is a simplex of \(K_n\). By Proposition~\ref{prop:canonical-star-top-cover}, every simplex of \(K_n\) is contained in some member of \(\mathcal C_n\). Thus there exists \(U\in\mathcal C_n\) with
\[
S\subseteq U.
\]
Then \(U\in V_{\lambda_i}\) for every \(i\), so \(U\) is a vertex of
\[
A_{\lambda_0}\cap\cdots\cap A_{\lambda_p}.
\]
Hence the intersection is nonempty.

Finally, assume the intersection is nonempty. Its vertices are precisely those members \(U\in\mathcal C_n\) that contain all \(\lambda_0,\dots,\lambda_p\). Any finite subfamily of such vertices still has nonempty common intersection containing \(S\), so it spans a simplex of the nerve. Therefore the whole intersection is itself a simplex of \(N_n\).
\end{proof}

\begin{corollary}\label{cor:anchor-cover-good}
The family \(\mathcal A_n\) is a good cover of \(N_n\) in the simplicial sense, and its nerve is canonically isomorphic to \(K_n\).
\end{corollary}

\begin{proof}
By Lemma~\ref{lem:anchor-cover}, the family \(\mathcal A_n\) covers \(N_n\). Each \(A_\lambda\) is a simplex of \(N_n\), hence contractible. By Proposition~\ref{prop:anchor-intersections-basic}, every nonempty finite intersection of members of \(\mathcal A_n\) is again a simplex of \(N_n\), hence contractible. Therefore \(\mathcal A_n\) is a good cover.

Its nerve has vertices indexed by the partitions \(\lambda\vdash n\), and by Proposition~\ref{prop:anchor-intersections-basic} a finite set
\[
\{A_{\lambda_0},\dots,A_{\lambda_p}\}
\]
spans a simplex of the nerve if and only if
\[
\{\lambda_0,\dots,\lambda_p\}
\]
is a simplex of \(K_n\). Thus the correspondence \(A_\lambda\leftrightarrow \lambda\) is a simplicial isomorphism from \(N(\mathcal A_n)\) to \(K_n\).
\end{proof}

\subsection{A rigidity lemma for higher intersections}

We now analyze the nonempty intersections
\[
A_S:=\bigcap_{\lambda\in S}A_\lambda
\]
for simplices \(S\subseteq K_n\).

Let \(\lambda\vdash n\), and let \(c\) be a removable corner of \(\lambda\). Recall that
\[
F_\lambda(c)=\{\lambda(c\to a):a\in A_{\max}(\lambda,c)\}
\]
is the corresponding star fiber, and denote by
\[
S_{\lambda,c}:=\Sigma^{\mathrm{star}}_{\max}(\lambda,c)
\]
the corresponding full star-simplex. Dually, for an addable corner \(a\) of \(\lambda\), write
\[
G_\lambda(a)=\{\lambda(c\to a):c\in C_{\max}(\lambda,a)\},
\qquad
T_{\lambda,a}:=\Sigma^{\mathrm{top}}_{\max}(\lambda,a).
\]

\begin{lemma}\label{lem:two-star-witnesses-force-apex}
Let \(U\in\mathcal C_n\) contain \(\lambda\), and suppose that, relative to \(\lambda\), the simplex \(U\) is star-type with support \(c\). Assume moreover that \(U\) contains two distinct vertices
\[
\mu_1=\lambda(c\to a_1),
\qquad
\mu_2=\lambda(c\to a_2),
\qquad
a_1\neq a_2.
\]
Then
\[
U=S_{\lambda,c}.
\]
\end{lemma}

\begin{proof}
Since \(U\in\mathcal C_n\), it is either a full star-simplex or a full top-simplex.

First suppose that \(U\) is a full top-simplex:
\[
U=\Sigma^{\mathrm{top}}_{\max}(\eta,b)
\]
for some partition \(\eta\) and some addable corner \(b\) of \(\eta\). Passing to conjugate partitions, there exist removable columns \(p,q_1,q_2\) of \(\eta\) such that
\[
\lambda'=\eta'-e_p+e_b,
\qquad
\mu_i'=\eta'-e_{q_i}+e_b
\quad (i=1,2).
\]
Subtracting the equation for \(\lambda'\) from that for \(\mu_i'\), we obtain
\[
\mu_i'-\lambda'=e_p-e_{q_i}.
\]
On the other hand, since \(\mu_i=\lambda(c\to a_i)\),
\[
\mu_i'-\lambda'=-e_{k(c)}+e_{\ell(a_i)}.
\]
Thus
\[
e_p-e_{q_i}=-e_{k(c)}+e_{\ell(a_i)},
\]
so necessarily
\[
p=\ell(a_i),\qquad q_i=k(c).
\]
This must hold for both \(i=1,2\), forcing \(\ell(a_1)=\ell(a_2)\), hence \(a_1=a_2\), contradiction. Therefore \(U\) cannot be a full top-simplex.

So \(U\) must be a full star-simplex:
\[
U=\Sigma^{\mathrm{star}}_{\max}(\nu,d)
\]
for some partition \(\nu\) and some removable corner \(d\) of \(\nu\). Again passing to conjugate coordinates, there exist addable columns \(p_0,p_1,p_2\) of \(\nu\) such that
\[
\lambda'=\nu'-e_{k(d)}+e_{p_0},
\qquad
\mu_i'=\nu'-e_{k(d)}+e_{p_i}
\quad (i=1,2).
\]
Subtracting the first equation from the others gives
\[
\mu_i'-\lambda'=-e_{p_0}+e_{p_i}.
\]
Comparing with
\[
\mu_i'-\lambda'=-e_{k(c)}+e_{\ell(a_i)},
\]
we get
\[
p_0=k(c),
\qquad
p_i=\ell(a_i),
\]
and therefore
\[
\nu'=\lambda'-e_{k(c)}+e_{k(d)}.
\]

Now let \(x\) be any vertex of \(U\). Then
\[
x'=\nu'-e_{k(d)}+e_j
\]
for some addable column \(j\) of \(\nu\). Substituting the expression for \(\nu'\), we obtain
\[
x'=\lambda'-e_{k(c)}+e_j.
\]
Hence \(x=\lambda\) when \(j=k(c)\), and otherwise \(x=\lambda(c\to a)\) for the addable corner \(a\) of \(\lambda\) with \(\ell(a)=j\). In particular, every vertex of \(U\) lies in
\[
\{\lambda\}\cup F_\lambda(c)=S_{\lambda,c}.
\]

Conversely, if \(y\in S_{\lambda,c}\), then either \(y=\lambda\) or
\[
y'=\lambda'-e_{k(c)}+e_j
\]
for some admissible target column \(j\). But then
\[
y'=\nu'-e_{k(d)}+e_j,
\]
so \(y\in U\). Therefore \(U=S_{\lambda,c}\).
\end{proof}

\begin{lemma}\label{lem:two-top-witnesses-force-apex}
Let \(U\in\mathcal C_n\) contain \(\lambda\), and suppose that, relative to \(\lambda\), the simplex \(U\) is top-type with support \(a\). Assume moreover that \(U\) contains two distinct vertices
\[
\nu_1=\lambda(c_1\to a),
\qquad
\nu_2=\lambda(c_2\to a),
\qquad
c_1\neq c_2.
\]
Then
\[
U=T_{\lambda,a}.
\]
\end{lemma}

\begin{proof}
Since \(U\in\mathcal C_n\), it is either a full top-simplex or a full star-simplex.

First suppose that \(U\) is a full star-simplex:
\[
U=\Sigma^{\mathrm{star}}_{\max}(\eta,b)
\]
for some partition \(\eta\) and some removable corner \(b\) of \(\eta\). Passing to conjugate partitions, there exist addable columns \(p,q_1,q_2\) of \(\eta\) such that
\[
\lambda'=\eta'-e_{k(b)}+e_p,
\qquad
\nu_i'=\eta'-e_{k(b)}+e_{q_i}
\quad (i=1,2).
\]
Subtracting the equation for \(\lambda'\) from that for \(\nu_i'\), we obtain
\[
\nu_i'-\lambda'=-e_p+e_{q_i}.
\]
On the other hand, since \(\nu_i=\lambda(c_i\to a)\),
\[
\nu_i'-\lambda'=e_{k(c_i)}-e_{\ell(a)}.
\]
Thus
\[
-e_p+e_{q_i}=e_{k(c_i)}-e_{\ell(a)},
\]
so necessarily
\[
p=\ell(a),\qquad q_i=k(c_i).
\]
This must hold for both \(i=1,2\), forcing \(k(c_1)=k(c_2)\), hence \(c_1=c_2\), contradiction. Therefore \(U\) cannot be a full star-simplex.

So \(U\) must be a full top-simplex:
\[
U=\Sigma^{\mathrm{top}}_{\max}(\eta,b)
\]
for some partition \(\eta\) and some addable corner \(b\) of \(\eta\). Again passing to conjugate coordinates, there exist removable columns \(p_0,p_1,p_2\) of \(\eta\) such that
\[
\lambda'=\eta'-e_{p_0}+e_{\ell(b)},
\qquad
\nu_i'=\eta'-e_{p_i}+e_{\ell(b)}
\quad (i=1,2).
\]
Subtracting the first equation from the others gives
\[
\nu_i'-\lambda'=e_{p_0}-e_{p_i}.
\]
Comparing with
\[
\nu_i'-\lambda'=e_{k(c_i)}-e_{\ell(a)},
\]
we get
\[
p_0=\ell(a),\qquad p_i=k(c_i),
\]
and therefore
\[
\eta'=\lambda'+e_{\ell(a)}-e_{\ell(b)}.
\]

Now let \(x\) be any vertex of \(U\). Then
\[
x'=\eta'-e_j+e_{\ell(b)}
\]
for some removable column \(j\) of \(\eta\). Substituting the expression for \(\eta'\), we obtain
\[
x'=\lambda'+e_{\ell(a)}-e_j.
\]
Hence \(x=\lambda\) when \(j=\ell(a)\), and otherwise \(x=\lambda(c\to a)\) for the removable corner \(c\) of \(\lambda\) with \(k(c)=j\). In particular, every vertex of \(U\) lies in
\[
\{\lambda\}\cup G_\lambda(a)=T_{\lambda,a}.
\]

Conversely, if \(y\in T_{\lambda,a}\), then either \(y=\lambda\) or
\[
y'=\lambda'+e_{\ell(a)}-e_j
\]
for some admissible removable column \(j\). But then
\[
y'=\eta'-e_j+e_{\ell(b)},
\]
so \(y\in U\). Therefore \(U=T_{\lambda,a}\).
\end{proof}

\begin{corollary}\label{cor:vertices-of-anchor-simplex}
Every vertex \(U\) of the anchor simplex \(A_\lambda\) is of one of the following three kinds:
\begin{enumerate}
    \item a full star apex \(S_{\lambda,c}\) for some removable corner \(c\) of \(\lambda\);
    \item a full top apex \(T_{\lambda,a}\) for some addable corner \(a\) of \(\lambda\);
    \item an edge vertex
    \[
    E_{\lambda,\mu}:=\{\lambda,\mu\},
    \]
    where \(\mu\sim\lambda\) and this edge itself belongs to \(\mathcal C_n\).
\end{enumerate}
\end{corollary}

\begin{proof}
Let \(U\) be a vertex of \(A_\lambda\). By Theorem~\ref{thm:clique-classification}, the simplex \(U\) is either star-type or top-type relative to \(\lambda\).

If \(U\) is star-type relative to \(\lambda\) with support \(c\) and contains at least two vertices of \(F_\lambda(c)\), then Lemma~\ref{lem:two-star-witnesses-force-apex} gives
\[
U=S_{\lambda,c}.
\]
If it contains only one such vertex besides \(\lambda\), then \(U\) is an edge \(\{\lambda,\mu\}\).

The top case is dual, using Lemma~\ref{lem:two-top-witnesses-force-apex}.
\end{proof}

\begin{theorem}\label{thm:anchor-intersection-classification}
Let
\[
S=\{\lambda_0,\dots,\lambda_p\}
\]
be a simplex of \(K_n\), and write
\[
A_S:=A_{\lambda_0}\cap\cdots\cap A_{\lambda_p}.
\]
Then:
\begin{enumerate}
    \item if \(|S|=1\), say \(S=\{\lambda\}\), then
    \[
    A_S=A_\lambda,
    \]
    a simplex of \(N_n\) of arbitrary dimension;
    \item if \(|S|=2\), say \(S=\{\lambda,\mu\}\) with \(\lambda\sim\mu\), then
    \[
    A_S=A_\lambda\cap A_\mu
    \]
    is a simplex of \(N_n\) with at most three vertices, hence of dimension at most \(2\);
    \item if \(|S|\ge 3\), then \(A_S\) is a single vertex of \(N_n\).
\end{enumerate}
\end{theorem}

\begin{proof}
The first statement is immediate from the definition.

For the second, let \(\mu=\lambda(c_\mu\to a_\mu)\). If \(U\) is a vertex of \(A_\lambda\cap A_\mu\), then \(U\) is a vertex of \(A_\lambda\). Since \(\mu\in U\), the three cases of Corollary~\ref{cor:vertices-of-anchor-simplex} specialize as follows:
\begin{enumerate}
    \item in the star-apex case, necessarily \(c=c_\mu\);
    \item in the top-apex case, necessarily \(a=a_\mu\);
    \item in the edge case, necessarily \(\nu=\mu\).
\end{enumerate}
Hence every vertex of \(A_\lambda\cap A_\mu\) belongs to the set
\[
\{\,S_{\lambda,c_\mu},\ T_{\lambda,a_\mu},\ E_{\lambda,\mu}\,\},
\]
with the understanding that some of these three vertices may be absent if they do not belong to \(\mathcal C_n\). Therefore \(A_\lambda\cap A_\mu\) has at most three vertices.

For the third statement, let \(|S|\ge 3\), and let \(\lambda\) be the unique vertex of minimal height in the simplex \(S\subseteq K_n\). By Theorem~\ref{thm:clique-classification}, the simplex \(S\) is either star-type or top-type relative to \(\lambda\).

Assume first that \(S\) is star-type relative to \(\lambda\), with support \(c\). Then \(S\) contains at least two distinct vertices
\[
\mu_1=\lambda(c\to a_1),
\qquad
\mu_2=\lambda(c\to a_2),
\qquad
a_1\neq a_2.
\]

Let \(U\) be a vertex of \(A_S\). Then \(U\in\mathcal C_n\) contains every vertex of \(S\), and in particular
\[
\lambda,\mu_1,\mu_2\in U.
\]
Since \(U\) is a simplex of \(K_n\) containing \(\lambda\), Theorem~\ref{thm:clique-classification} applies to \(U\) relative to \(\lambda\): the simplex \(U\) is either star-type or top-type relative to \(\lambda\).

It cannot be top-type relative to \(\lambda\), because a top-type simplex through \(\lambda\) cannot contain two vertices of the form
\[
\lambda(c\to a_1),\qquad \lambda(c\to a_2)
\]
with \(a_1\neq a_2\). Therefore \(U\) is star-type relative to \(\lambda\).

Moreover, in a star-type simplex through \(\lambda\), all noncentral vertices have the same removable corner. Since both \(\mu_1\) and \(\mu_2\) belong to \(U\), that common removable corner must be \(c\). Thus \(U\) is star-type relative to \(\lambda\) with support \(c\). By Lemma~\ref{lem:two-star-witnesses-force-apex},
\[
U=S_{\lambda,c}.
\]

Hence every vertex of \(A_S\) is the same vertex \(S_{\lambda,c}\), and therefore
\[
A_S=\{S_{\lambda,c}\}.
\]

The top-type case is dual. If \(S\) is top-type relative to \(\lambda\), with support \(a\), then \(S\) contains two distinct vertices
\[
\nu_1=\lambda(c_1\to a),
\qquad
\nu_2=\lambda(c_2\to a),
\qquad
c_1\neq c_2.
\]
For any \(U\in A_S\), the simplex \(U\) contains \(\lambda,\nu_1,\nu_2\), so it cannot be star-type relative to \(\lambda\). Hence it is top-type relative to \(\lambda\), and its common addable corner must be \(a\). By Lemma~\ref{lem:two-top-witnesses-force-apex},
\[
U=T_{\lambda,a}.
\]
Thus
\[
A_S=\{T_{\lambda,a}\}.
\]

This proves the third statement.
\end{proof}

\subsection{The intersection poset of the anchor cover}

Let
\[
\mathcal P_n:=\{\text{nonempty simplices of }K_n\},
\]
ordered by inclusion. Define, for every \(S\in\mathcal P_n\),
\[
A_S:=\bigcap_{\lambda\in S} A_\lambda.
\]

We now introduce a closure operator on \(\mathcal P_n\), induced by the anchor cover.

\begin{definition}\label{def:anchor-closure}
For \(S\in\mathcal P_n\), define
\[
\operatorname{cl}_{\mathcal A}(S)
:=
\{\mu\in K_n^{(0)} : A_S\subseteq A_\mu\}.
\]
By Proposition~\ref{prop:anchor-closure-operator}, this set is itself a simplex of \(K_n\), so \(\operatorname{cl}_{\mathcal A}(S)\in\mathcal P_n\).
\end{definition}

\begin{proposition}\label{prop:anchor-closure-operator}
The map
\[
\operatorname{cl}_{\mathcal A}:\mathcal P_n\to \mathcal P_n
\]
is a closure operator.
\end{proposition}

\begin{proof}
We verify the three defining properties.

\emph{Extensive.} If \(\lambda\in S\), then by definition
\[
A_S\subseteq A_\lambda,
\]
so \(\lambda\in \operatorname{cl}_{\mathcal A}(S)\). Thus
\[
S\subseteq \operatorname{cl}_{\mathcal A}(S).
\]

\emph{Monotone.} If \(S\subseteq T\), then
\[
A_T\subseteq A_S.
\]
Therefore, if \(\mu\in \operatorname{cl}_{\mathcal A}(S)\), so that \(A_S\subseteq A_\mu\), then also
\[
A_T\subseteq A_S\subseteq A_\mu,
\]
hence \(\mu\in \operatorname{cl}_{\mathcal A}(T)\). So
\[
\operatorname{cl}_{\mathcal A}(S)\subseteq \operatorname{cl}_{\mathcal A}(T).
\]

\emph{Idempotent.} Set \(C:=\operatorname{cl}_{\mathcal A}(S)\). Since \(S\subseteq C\), we have
\[
A_C\subseteq A_S.
\]
Conversely, for every \(\mu\in C\) we have \(A_S\subseteq A_\mu\), so
\[
A_S\subseteq \bigcap_{\mu\in C} A_\mu = A_C.
\]
Hence
\[
A_C=A_S.
\]
Now
\[
\operatorname{cl}_{\mathcal A}(C)
=
\{\mu: A_C\subseteq A_\mu\}
=
\{\mu: A_S\subseteq A_\mu\}
=
\operatorname{cl}_{\mathcal A}(S)
=
C.
\]
Therefore \(\operatorname{cl}_{\mathcal A}\) is idempotent.

Finally, if \(\mu,\nu\in \operatorname{cl}_{\mathcal A}(S)\), then
\[
A_S\subseteq A_\mu\cap A_\nu,
\]
so \(A_{\{\mu,\nu\}}\neq\varnothing\). By Proposition~\ref{prop:anchor-intersections-basic}, the set \(\{\mu,\nu\}\) is a simplex of \(K_n\). Thus \(\operatorname{cl}_{\mathcal A}(S)\) is itself a simplex of \(K_n\), i.e. an element of \(\mathcal P_n\).
\end{proof}

Let
\[
\mathcal J_n:=\{\,A_S : S\in\mathcal P_n\,\},
\]
where equal intersections are identified, and order \(\mathcal J_n\) by inclusion.

\begin{proposition}\label{prop:closed-simplices-equal-intersections}
The assignment
\[
S\longmapsto A_S
\]
induces an order-reversing bijection between the fixed simplices of the closure operator \(\operatorname{cl}_{\mathcal A}\) and the poset \(\mathcal J_n\).

Equivalently, if
\[
\operatorname{Fix}(\operatorname{cl}_{\mathcal A})
:=
\{\,S\in\mathcal P_n : \operatorname{cl}_{\mathcal A}(S)=S\,\},
\]
then
\[
\operatorname{Fix}(\operatorname{cl}_{\mathcal A})^{\mathrm{op}}
\cong
\mathcal J_n.
\]
\end{proposition}

\begin{proof}
For every \(X\in\mathcal J_n\), define
\[
\Psi(X):=\{\mu\in K_n^{(0)} : X\subseteq A_\mu\}.
\]
Since \(X=A_S\) for some simplex \(S\in\mathcal P_n\), Proposition~\ref{prop:anchor-closure-operator} implies that \(\Psi(X)\) is a simplex of \(K_n\), i.e. an element of \(\mathcal P_n\).

Now let \(S\in\mathcal P_n\). Then
\[
\Psi(A_S)
=
\{\mu : A_S\subseteq A_\mu\}
=
\operatorname{cl}_{\mathcal A}(S).
\]
In particular, if \(S\) is fixed by \(\operatorname{cl}_{\mathcal A}\), then
\[
\Psi(A_S)=S.
\]

Conversely, let \(X\in\mathcal J_n\), and write \(X=A_S\) for some simplex \(S\in\mathcal P_n\). Set
\[
T:=\Psi(X)=\operatorname{cl}_{\mathcal A}(S).
\]
Since \(S\subseteq T\), we have
\[
A_T\subseteq A_S=X.
\]
On the other hand, by definition of \(T\), every \(\mu\in T\) satisfies
\[
X\subseteq A_\mu.
\]
Therefore
\[
X\subseteq \bigcap_{\mu\in T} A_\mu = A_T.
\]
Hence
\[
A_T=X.
\]

So the two assignments
\[
S\longmapsto A_S,
\qquad
X\longmapsto \Psi(X)
\]
are inverse to one another between \(\operatorname{Fix}(\operatorname{cl}_{\mathcal A})\) and \(\mathcal J_n\).

Finally, if \(S\subseteq T\), then
\[
A_T\subseteq A_S,
\]
so the correspondence is order-reversing. This proves the claimed anti-isomorphism.
\end{proof}

\begin{theorem}\label{thm:Kn-equivalent-order-complex-Jn}
There is a homotopy equivalence
\[
K_n \simeq \Delta(\mathcal J_n).
\]
Consequently,
\[
N_n \simeq \Delta(\mathcal J_n).
\]
\end{theorem}

\begin{proof}
A standard result on closure operators on finite posets implies that the order complex of \(\mathcal P_n\) deformation retracts onto the order complex of the fixed part of \(\operatorname{cl}_{\mathcal A}\); see, for instance, \cite{Kozlov2008}. Therefore
\[
\Delta(\mathcal P_n)\simeq \Delta(\operatorname{Fix}(\operatorname{cl}_{\mathcal A})).
\]
Since \(\mathcal P_n\) is the poset of nonempty simplices of \(K_n\),
\[
\Delta(\mathcal P_n)=sd(K_n)\simeq K_n.
\]
By Proposition~\ref{prop:closed-simplices-equal-intersections}, the fixed part of \(\operatorname{cl}_{\mathcal A}\) is anti-isomorphic to \(\mathcal J_n\). Since a poset and its opposite have canonically isomorphic order complexes,
\[
\Delta(\operatorname{Fix}(\operatorname{cl}_{\mathcal A}))\cong \Delta(\mathcal J_n).
\]
Hence
\[
K_n \simeq \Delta(\mathcal J_n).
\]

Finally, by Corollary~\ref{cor:Kn-nerve-equivalence},
\[
K_n\simeq N_n.
\]
Therefore
\[
N_n \simeq \Delta(\mathcal J_n).
\]
\end{proof}

\subsection{Chain length in \texorpdfstring{\(\mathcal J_n\)}{J\_n}}

We now use Theorem~\ref{thm:anchor-intersection-classification} to control the dimension of \(\Delta(\mathcal J_n)\).

\begin{lemma}\label{lem:non-singleton-elements-Jn}
Every non-singleton element of \(\mathcal J_n\) is either:
\begin{enumerate}
    \item an anchor simplex \(A_\lambda\), or
    \item a pairwise overlap \(A_{\lambda,\mu}:=A_\lambda\cap A_\mu\).
\end{enumerate}
\end{lemma}

\begin{proof}
Let \(X\in\mathcal J_n\) be non-singleton. Then \(X=A_S\) for some simplex \(S\subseteq K_n\). By Theorem~\ref{thm:anchor-intersection-classification}, if \(|S|\ge 3\), then \(A_S\) is a singleton, contradiction. Hence \(|S|\le 2\). If \(|S|=1\), then \(X=A_\lambda\); if \(|S|=2\), then \(X=A_{\lambda,\mu}\).
\end{proof}

\begin{lemma}\label{lem:inside-anchor}
Let \(X\in\mathcal J_n\) be non-singleton, and suppose
\[
X\subseteq A_\lambda.
\]
Then either
\[
X=A_\lambda,
\]
or
\[
X=A_{\lambda,\mu}
\]
for some neighbor \(\mu\sim\lambda\).
\end{lemma}

\begin{proof}
Write \(X=A_S\) for some simplex \(S\subseteq K_n\). Since \(X\subseteq A_\lambda\),
\[
X=A_S\cap A_\lambda = A_{S\cup\{\lambda\}}.
\]
Because \(X\) is non-singleton, Theorem~\ref{thm:anchor-intersection-classification} implies that \(|S\cup\{\lambda\}|\le 2\). Hence either \(S\cup\{\lambda\}=\{\lambda\}\), giving \(X=A_\lambda\), or \(S\cup\{\lambda\}=\{\lambda,\mu\}\), giving \(X=A_{\lambda,\mu}\).
\end{proof}

\begin{lemma}\label{lem:inside-pairwise-overlap}
Let \(X\in\mathcal J_n\) be non-singleton, and suppose
\[
X\subseteq A_{\lambda,\mu}:=A_\lambda\cap A_\mu.
\]
Then
\[
X=A_{\lambda,\mu}.
\]
\end{lemma}

\begin{proof}
Write \(X=A_S\) for some simplex \(S\subseteq K_n\). Since \(X\subseteq A_\lambda\cap A_\mu\),
\[
X=A_S\cap A_\lambda\cap A_\mu = A_{S\cup\{\lambda,\mu\}}.
\]
Again \(X\) is non-singleton, so by Theorem~\ref{thm:anchor-intersection-classification} the simplex \(S\cup\{\lambda,\mu\}\) has size at most \(2\). But it already contains \(\{\lambda,\mu\}\), hence
\[
S\cup\{\lambda,\mu\}=\{\lambda,\mu\},
\]
and therefore
\[
X=A_{\lambda,\mu}.
\]
\end{proof}

\begin{theorem}\label{thm:J_n-dimension-at-most-2}
Every strict inclusion chain in \(\mathcal J_n\) has length at most \(2\). Equivalently,
\[
\dim \Delta(\mathcal J_n)\le 2.
\]
\end{theorem}

\begin{proof}
Take a strict inclusion chain
\[
X_0 \subsetneq X_1 \subsetneq \cdots \subsetneq X_r
\qquad\text{in }\mathcal J_n.
\]

If \(X_r\) is a singleton, then \(r=0\).

Assume now that \(X_r\) is non-singleton. By Lemma~\ref{lem:non-singleton-elements-Jn}, it is either an anchor simplex \(A_\lambda\) or a pairwise overlap \(A_{\lambda,\mu}\).

If
\[
X_r=A_{\lambda,\mu},
\]
then by Lemma~\ref{lem:inside-pairwise-overlap}, every non-singleton element of \(\mathcal J_n\) contained in \(X_r\) must equal \(X_r\). Hence any element strictly below \(X_r\) is a singleton. So the chain has length at most \(1\):
\[
\{U\} \subsetneq A_{\lambda,\mu}.
\]

If
\[
X_r=A_\lambda,
\]
then by Lemma~\ref{lem:inside-anchor}, every proper non-singleton element below \(A_\lambda\) is of the form \(A_{\lambda,\mu}\). By Lemma~\ref{lem:inside-pairwise-overlap}, nothing non-singleton lies properly below such a pairwise overlap. Therefore the longest possible chain is
\[
\{U\} \subsetneq A_{\lambda,\mu} \subsetneq A_\lambda,
\]
which has length \(2\).

Thus in all cases
\[
r\le 2.
\]
Hence
\[
\dim \Delta(\mathcal J_n)\le 2.
\]
\end{proof}

\subsection{A \texorpdfstring{\(2\)}{2}-dimensional model}

We can now complete the reduction step needed for the main theorem.

\begin{theorem}\label{thm:Nn-has-2-dimensional-model}
The nerve \(N_n\) has the homotopy type of a CW-complex of dimension at most \(2\).
\end{theorem}

\begin{proof}
By Theorem~\ref{thm:Kn-equivalent-order-complex-Jn},
\[
N_n \simeq \Delta(\mathcal J_n).
\]
By Theorem~\ref{thm:J_n-dimension-at-most-2},
\[
\dim \Delta(\mathcal J_n)\le 2.
\]
Therefore \(N_n\) has the homotopy type of a simplicial complex, and hence of a CW-complex, of dimension at most \(2\).
\end{proof}

\begin{corollary}\label{cor:Kn-has-2-dimensional-model}
The clique complex \(K_n\) has the homotopy type of a CW-complex of dimension at most \(2\).
\end{corollary}

\begin{proof}
By Corollary~\ref{cor:Kn-nerve-equivalence},
\[
K_n\simeq N_n.
\]
Now apply Theorem~\ref{thm:Nn-has-2-dimensional-model}.
\end{proof}

\begin{corollary}\label{cor:high-homology-vanishes}
For every \(k\ge 3\),
\[
\widetilde H_k(K_n)=0.
\]
\end{corollary}

\begin{proof}
A CW-complex of dimension at most \(2\) has trivial reduced homology in all degrees \(k\ge 3\). Apply Corollary~\ref{cor:Kn-has-2-dimensional-model} and homotopy invariance of homology.
\end{proof}

\subsection{Summary}

We have replaced the direct collapse argument by a two-step reduction. First, the anchor cover \(\mathcal A_n\) of the nerve \(N_n\) gives rise to an intersection poset \(\mathcal J_n\) such that
\[
N_n \simeq \Delta(\mathcal J_n).
\]
Second, the classification of the nonempty intersections of anchor simplices shows that every strict chain in \(\mathcal J_n\) has length at most \(2\). Hence
\[
\dim \Delta(\mathcal J_n)\le 2,
\]
and therefore both \(N_n\) and \(K_n\) have the homotopy type of CW-complexes of dimension at most \(2\).

\section{Connectedness and simple connectedness}
\label{sec:connectedness-and-simple-connectedness}

In the previous section we proved that the nerve \(N_n\) of the canonical cover \(\mathcal C_n\) has the homotopy type of a CW-complex of dimension at most \(2\). Since \(K_n\simeq N_n\), the same holds for \(K_n\). We now supply the complementary fundamental-group statement by proving that \(K_n\) is connected and simply connected.

\subsection{Connectedness}

We begin by proving connectedness.

\begin{proposition}\label{prop:Kn-connected}
The graph \(G_n\), and hence the clique complex \(K_n\), is connected.
\end{proposition}

\begin{proof}
It suffices to show that \(G_n\) is connected. Let
\[
\lambda=(\lambda_1,\dots,\lambda_\ell)\vdash n.
\]
If \(\lambda=(n)\), the claim is immediate. Otherwise, \(\lambda\) has at least two nonzero parts. By repeatedly transferring one cell from the last nonzero row to the first row, and reordering after each step, we obtain a path in \(G_n\) from \(\lambda\) to the one-row partition \((n)\). Hence every vertex of \(G_n\) is connected to \((n)\), so \(G_n\) is connected. Therefore \(K_n\) is connected as well.
\end{proof}

\subsection{The height function}

Recall from Section~\ref{sec:partitions-moves-canonical-simplices} the height function
\[
h(\lambda):=\sum_{i=1}^{\ell} i\,\lambda_i.
\]

We shall use this function to simplify edge-loops in the \(1\)-skeleton of \(K_n\), which is precisely the graph \(G_n\).

\begin{lemma}\label{lem:height-change-recalled}
If
\[
\mu=\lambda(c\to a),
\]
then
\[
h(\mu)-h(\lambda)=r(a)-r(c).
\]
In particular, if \(r(a)<r(c)\), then
\[
h(\mu)<h(\lambda).
\]
\end{lemma}

\begin{proof}
This is Lemma~\ref{lem:height-change}.
\end{proof}

A key consequence is that adjacent vertices always have different heights.

\begin{lemma}\label{lem:adjacent-heights-different-recalled}
If \(\lambda\) and \(\mu\) are distinct adjacent vertices of \(G_n\), then
\[
h(\lambda)\neq h(\mu).
\]
\end{lemma}

\begin{proof}
This is Lemma~\ref{lem:adjacent-different-heights}.
\end{proof}

\subsection{Loop complexity}

Let
\[
L=(\lambda_0,\lambda_1,\dots,\lambda_m=\lambda_0)
\]
be an edge-loop in \(K_n\). Define
\[
H(L):=\max_i h(\lambda_i),
\]
and
\[
M(L):=\#\{\,i:h(\lambda_i)=H(L)\,\}.
\]
We order such pairs lexicographically:
\[
(H',M')<(H,M)
\]
if either \(H'<H\), or \(H'=H\) and \(M'<M\).

\subsection{Local peak reduction}

We now recall the local combinatorial lemma established earlier.

\begin{lemma}\label{lem:local-peak-reduction-final}
Let
\[
\mu_1-\lambda-\mu_2
\]
be a length-two fragment of an edge-loop in \(K_n\) such that
\[
h(\lambda)>h(\mu_1),\qquad h(\lambda)>h(\mu_2).
\]
Then this fragment is homotopic in \(K_n\), relative to its endpoints, to an edge-path with the same endpoints and all intermediate vertices of height strictly smaller than \(h(\lambda)\).
\end{lemma}

\begin{proof}
Write
\[
\mu_1=\lambda(c_1\to a_1),\qquad
\mu_2=\lambda(c_2\to a_2).
\]
Since \(h(\mu_i)<h(\lambda)\), Lemma~\ref{lem:height-change-recalled} implies
\[
r(a_1)<r(c_1),\qquad r(a_2)<r(c_2).
\]

If \(c_1=c_2\) or \(a_1=a_2\), then \(\lambda,\mu_1,\mu_2\) span a triangle in \(K_n\) by Lemmas~\ref{lem:star-triangle-criterion} and \ref{lem:top-triangle-criterion}. Hence the path
\[
\mu_1-\lambda-\mu_2
\]
is homotopic, relative to its endpoints, to the single edge
\[
\mu_1-\mu_2.
\]

Assume now that \(c_1\neq c_2\) and \(a_1\neq a_2\). Without loss of generality, reorder indices so that
\[
r(c_1)\le r(c_2).
\]
Define
\[
\nu:=\lambda(c_2\to a_1).
\]

We claim that this move is admissible. Indeed, removing the removable corner \(c_2\) preserves the Ferrers condition. Moreover, since
\[
r(a_1)<r(c_1)\le r(c_2),
\]
all rows weakly above \(r(a_1)\) are unchanged when the cell \(c_2\) is removed. Hence the local Ferrers conditions that make \(a_1\) addable remain valid after the removal of \(c_2\). Therefore \(a_1\) is still addable in the intermediate diagram, and the partition \(\nu\) is well defined.

Now \(\mu_1\) and \(\nu\) share the same addable corner \(a_1\), so \(\lambda,\mu_1,\nu\) span a top-triangle by Lemma~\ref{lem:top-triangle-criterion}. Similarly, \(\nu\) and \(\mu_2\) share the same removable corner \(c_2\), so \(\lambda,\nu,\mu_2\) span a star-triangle by Lemma~\ref{lem:star-triangle-criterion}. In particular, \(\mu_1\) is adjacent to \(\nu\) and \(\nu\) is adjacent to \(\mu_2\). Therefore
\[
\mu_1-\lambda-\mu_2
\]
is homotopic in \(K_n\) to
\[
\mu_1-\nu-\mu_2.
\]

Finally, the move \(\lambda\to \nu\) sends a cell from row \(r(c_2)\) to row \(r(a_1)\), and
\[
r(a_1)<r(c_2),
\]
so Lemma~\ref{lem:height-change-recalled} gives
\[
h(\nu)<h(\lambda).
\]
Thus every intermediate vertex in the replacement path has height strictly smaller than \(h(\lambda)\), as required.
\end{proof}

\subsection{Maximal-height vertices are isolated peaks}

The height function has a decisive simplifying feature: adjacent vertices always have different heights. Hence every maximal-height vertex along a loop is automatically isolated.

\begin{proposition}\label{prop:max-height-isolated}
Let \(L\) be an edge-loop in \(K_n\). Every vertex of \(L\) of height \(H(L)\) has both neighboring vertices of strictly smaller height.
\end{proposition}

\begin{proof}
Let \(\lambda_i\) be a vertex of \(L\) such that
\[
h(\lambda_i)=H(L).
\]
Its neighbors in the loop are \(\lambda_{i-1}\) and \(\lambda_{i+1}\). By Lemma~\ref{lem:adjacent-heights-different-recalled}, neither of these can have the same height as \(\lambda_i\). Since \(H(L)\) is maximal along the loop, both must have strictly smaller height.
\end{proof}

\subsection{Reduction of edge-loops}

We can now reduce any nonconstant loop.

\begin{proposition}\label{prop:loop-reduction}
Every nonconstant edge-loop \(L\) in \(K_n\) is homotopic to an edge-loop \(L'\) with
\[
(H(L'),M(L'))<(H(L),M(L)).
\]
\end{proposition}

\begin{proof}
Choose a vertex \(\lambda\) of \(L\) with height \(H(L)\). By Proposition~\ref{prop:max-height-isolated}, the two neighboring vertices of \(\lambda\) along the loop have strictly smaller height. Hence the corresponding length-two fragment satisfies the hypotheses of Lemma~\ref{lem:local-peak-reduction-final}. Replacing it by the homotopic path given by that lemma produces a new loop \(L'\).

By construction, every newly introduced intermediate vertex has height strictly smaller than \(H(L)\). Thus at least one occurrence of height \(H(L)\) disappears and no new one is created. Therefore
\[
(H(L'),M(L'))<(H(L),M(L))
\]
in lexicographic order.
\end{proof}

\subsection{Simple connectedness}

We may now complete the proof.

\begin{theorem}\label{thm:Kn-simply-connected}
The clique complex \(K_n\) is simply connected.
\end{theorem}

\begin{proof}
By Proposition~\ref{prop:Kn-connected}, the complex \(K_n\) is connected. It therefore suffices to show that every edge-loop in \(K_n\) is null-homotopic.

We argue by well-founded induction on the lexicographically ordered pair \((H(L),M(L))\). This order is well founded because \(h\) takes integer values on the finite set \(\Par(n)\). Assume, for contradiction, that there exists a non-null-homotopic edge-loop, and choose one for which \((H(L),M(L))\) is minimal.

Since \(L\) is nonconstant, Proposition~\ref{prop:loop-reduction} produces a homotopic edge-loop \(L'\) such that
\[
(H(L'),M(L'))<(H(L),M(L)).
\]
By minimality, \(L'\) is null-homotopic. Hence \(L\) is null-homotopic as well, a contradiction.

Therefore every edge-loop in \(K_n\) is null-homotopic, and so
\[
\pi_1(K_n)=0.
\]
\end{proof}

\begin{corollary}\label{cor:Kn-connected-simply-connected}
The clique complex \(K_n\) is connected and simply connected.
\end{corollary}

\begin{proof}
Combine Proposition~\ref{prop:Kn-connected} and Theorem~\ref{thm:Kn-simply-connected}.
\end{proof}

\subsection{Summary}

We have shown that \(K_n\) is connected and simply connected. Together with the \(2\)-dimensional model obtained in Section~\ref{sec:anchor-cover-2-model}, this provides the final topological input needed for the wedge-of-spheres theorem proved in the next section.

\section{The homotopy type of the clique complex}
\label{sec:homotopy-type}

We now combine the two topological inputs established in the previous sections. Section~\ref{sec:anchor-cover-2-model} showed that \(K_n\) has the homotopy type of a CW-complex of dimension at most \(2\), while Section~\ref{sec:connectedness-and-simple-connectedness} showed that \(K_n\) is connected and simply connected. It follows that the reduced homology of \(K_n\) is concentrated in degree \(2\), and this determines its homotopy type completely.

\subsection{A simply connected \texorpdfstring{\(2\)}{2}-dimensional model}

We first record the combined conclusion of Sections~\ref{sec:anchor-cover-2-model} and \ref{sec:connectedness-and-simple-connectedness}.

\begin{proposition}\label{prop:simply-connected-2-model}
The clique complex \(K_n\) has the homotopy type of a simply connected CW-complex of dimension at most \(2\).
\end{proposition}

\begin{proof}
By Corollary~\ref{cor:Kn-has-2-dimensional-model}, the complex \(K_n\) is homotopy equivalent to a CW-complex of dimension at most \(2\). By Corollary~\ref{cor:Kn-connected-simply-connected}, the complex \(K_n\) is connected and simply connected. Since connectedness and simple connectedness are homotopy invariants, the chosen CW-model is also connected and simply connected.
\end{proof}

\subsection{Homology and Euler characteristic}

We now determine the reduced homology of \(K_n\).

\begin{proposition}\label{prop:H2-rank}
The reduced homology of \(K_n\) is concentrated in degree \(2\), and
\[
\widetilde H_2(K_n)\cong \mathbb Z^{\,\chi(K_n)-1}.
\]
Equivalently,
\[
\widetilde H_i(K_n)=0 \quad (i\neq 2),
\qquad
\operatorname{rank}\widetilde H_2(K_n)=\chi(K_n)-1.
\]
\end{proposition}

\begin{proof}
Since \(K_n\) is connected, we have
\[
\widetilde H_0(K_n)=0.
\]
Since \(K_n\) is path-connected and simply connected, its first homology vanishes:
\[
H_1(K_n)=0.
\]
By Corollary~\ref{cor:high-homology-vanishes},
\[
\widetilde H_k(K_n)=0 \qquad (k\ge 3).
\]

Thus the only possibly nonzero reduced homology group is \(\widetilde H_2(K_n)\). Since \(K_n\) has the homotopy type of a simply connected CW-complex of dimension at most \(2\), the group \(\widetilde H_2(K_n)\) is free abelian. The reduced Euler characteristic therefore satisfies
\[
\widetilde\chi(K_n)=\operatorname{rank}\widetilde H_2(K_n).
\]
Using
\[
\widetilde\chi(K_n)=\chi(K_n)-1,
\]
we obtain
\[
\operatorname{rank}\widetilde H_2(K_n)=\chi(K_n)-1.
\]
This proves the claim.
\end{proof}

\subsection{A wedge lemma for simply connected \texorpdfstring{\(2\)}{2}-complexes}

For completeness, we isolate the standard fact needed to pass from homology to homotopy type; compare \cite[Chapter~4]{Hatcher2002}.

\begin{lemma}\label{lem:simply-connected-2-complex-wedge}
Let \(X\) be a connected, simply connected CW-complex of dimension at most \(2\). Then \(X\) is homotopy equivalent to a wedge of \(2\)-spheres. More precisely,
\[
X \simeq \bigvee^{\,r} S^2,
\qquad
r=\operatorname{rank} H_2(X).
\]
\end{lemma}

\begin{proof}
Since \(X\) is connected and simply connected, the Hurewicz theorem gives an isomorphism
\[
\pi_2(X)\cong H_2(X).
\]
Because \(X\) has dimension at most \(2\), the group \(H_2(X)\) is free abelian. Choose elements
\[
\alpha_1,\dots,\alpha_r\in \pi_2(X)
\]
whose Hurewicz images form a basis of \(H_2(X)\), where \(r=\operatorname{rank}H_2(X)\). These classes determine a map
\[
f:\bigvee^{\,r}S^2\longrightarrow X.
\]
By construction, \(f\) induces an isomorphism on \(H_2\). Both spaces are connected and simply connected, and both have trivial homology in degrees \(\ge 3\) because they have dimension at most \(2\). Hence \(f\) induces an isomorphism on all homology groups. By the Whitehead theorem for simply connected CW-complexes, \(f\) is a homotopy equivalence; compare \cite[Chapter~4]{Hatcher2002}. Therefore
\[
X\simeq \bigvee^{\,r} S^2.
\]
\end{proof}

\subsection{The main theorem}

We can now state and prove the principal result of the paper.

\begin{theorem}\label{thm:main-wedge}
For every \(n\ge 1\), the clique complex \(K_n\) is homotopy equivalent to a wedge of \(2\)-spheres:
\[
K_n \simeq \bigvee^{\,b_n} S^2,
\qquad
b_n=\chi(K_n)-1.
\]
\end{theorem}

\begin{proof}
By Proposition~\ref{prop:simply-connected-2-model}, the complex \(K_n\) has the homotopy type of a simply connected CW-complex of dimension at most \(2\). By Proposition~\ref{prop:H2-rank}, its reduced homology is concentrated in degree \(2\), where it is free abelian of rank \(\chi(K_n)-1\). Applying Lemma~\ref{lem:simply-connected-2-complex-wedge}, we obtain
\[
K_n \simeq \bigvee^{\,\chi(K_n)-1} S^2.
\]
\end{proof}

\subsection{Immediate consequences}

We record two direct corollaries.

\begin{corollary}\label{cor:homology-concentrated}
For every \(n\ge 1\),
\[
\widetilde H_i(K_n)=0 \quad (i\neq 2),
\qquad
\widetilde H_2(K_n)\cong \mathbb Z^{\,\chi(K_n)-1}.
\]
\end{corollary}

\begin{proof}
This is just Proposition~\ref{prop:H2-rank}.
\end{proof}

\begin{corollary}\label{cor:homotopy-determined-by-euler}
The homotopy type of \(K_n\) is completely determined by its Euler characteristic.
\end{corollary}

\begin{proof}
By Theorem~\ref{thm:main-wedge}, the homotopy type of \(K_n\) is that of a wedge of exactly \(\chi(K_n)-1\) copies of \(S^2\), and hence depends only on \(\chi(K_n)\).
\end{proof}

\begin{remark}\label{rem:high-dimensional-simplices-collapse}
Theorem~\ref{thm:main-wedge} exhibits a clear contrast between local combinatorics and global topology. Although the clique complex \(K_n\) contains simplices of arbitrarily high dimension as \(n\) grows, its homotopy type is always that of a wedge of \(2\)-spheres.
\end{remark}

\subsection{Summary}

The main theorem has now been proved: for every \(n\), the clique complex of the partition graph is homotopy equivalent to a wedge of \(2\)-spheres, and the number of spheres is \(\chi(K_n)-1\). The remaining task is therefore numerical rather than qualitative: to understand the Euler characteristic \(\chi(K_n)\), compute small examples, and record the resulting integer sequences. This is the subject of the final section.

\section{Euler characteristic, examples, and integer sequences}

The qualitative classification of the clique complex \(K_n\) has now been completed: by Theorem~\ref{thm:main-wedge}, it is always a wedge of \(2\)-spheres. Thus the remaining problem is numerical rather than qualitative. The number of spheres is
\[
b_n=\chi(K_n)-1,
\]
so understanding the Euler characteristic \(\chi(K_n)\) is equivalent to understanding the precise homotopy type of \(K_n\).

\subsection{Euler characteristic as the decisive invariant}

Since
\[
K_n \simeq \bigvee^{\,b_n} S^2,
\qquad
b_n=\chi(K_n)-1,
\]
the Euler characteristic completely determines the homotopy type of \(K_n\). In particular, any explicit formula, recurrence, or effective counting scheme for \(\chi(K_n)\) immediately yields the number of spheres in the wedge decomposition.

By definition,
\[
\chi(K_n)=\sum_{p\ge 0}(-1)^p f_p(K_n),
\]
where \(f_p(K_n)\) denotes the number of \(p\)-dimensional simplices of \(K_n\). Since simplices of \(K_n\) are exactly cliques in \(G_n\), we may also write
\[
\chi(K_n)=\sum_{r\ge 1}(-1)^{r-1} c_r(n),
\]
where \(c_r(n)\) is the number of cliques of size \(r\) in \(G_n\).

\begin{proposition}\label{prop:chi-counting-schemes}
The Euler characteristic \(\chi(K_n)\) may be computed either from the alternating sum of clique numbers in \(G_n\), or from the nerve \(N_n\) of the canonical cover. In particular,
\[
\chi(K_n)=\chi(N_n).
\]
\end{proposition}

\begin{proof}
The first formula is simply the definition of the Euler characteristic of a finite simplicial complex. The second follows from Corollary~\ref{cor:nerve-equivalence}, since Euler characteristic is a homotopy invariant for finite CW-complexes.
\end{proof}

\begin{remark}\label{rem:chi-two-viewpoints}
Proposition~\ref{prop:chi-counting-schemes} gives two complementary ways to approach \(\chi(K_n)\): directly, through clique counts in the partition graph, and indirectly, through the overlap combinatorics of the maximal simplices encoded by the nerve.
\end{remark}

\subsection{The numbers \texorpdfstring{\(b_n\)}{b\_n}}

The most natural numerical invariant arising from our theorem is the sequence
\[
b_n:=\chi(K_n)-1.
\]
By Theorem~\ref{thm:main-wedge} and Corollary~\ref{cor:homology-concentrated}, this integer has several equivalent interpretations:
\[
b_n=\operatorname{rank}\widetilde H_2(K_n),
\]
and it is also the number of \(2\)-spheres in the wedge decomposition of \(K_n\).

\begin{corollary}\label{cor:bn-interpretations}
For every \(n\ge 1\), the integer \(b_n\) is simultaneously:
\begin{enumerate}
    \item the number of \(2\)-spheres in the wedge decomposition of \(K_n\);
    \item the rank of the reduced homology group \(\widetilde H_2(K_n)\);
    \item the shifted Euler characteristic \(\chi(K_n)-1\).
\end{enumerate}
\end{corollary}

\begin{proof}
Immediate from Theorem~\ref{thm:main-wedge} and Proposition~\ref{prop:H2-rank}.
\end{proof}

\subsection{Small values}

To make the theorem concrete, we record the first values of \(\chi(K_n)\) and \(b_n\). These were obtained by direct computer enumeration of cliques in the partition graph \(G_n\).

\begin{table}[htbp]
\centering
\begin{tabular}{c|c|c}
\(n\) & \(\chi(K_n)\) & \(b_n=\chi(K_n)-1\)\\
\hline
1 & 1 & 0\\
2 & 1 & 0\\
3 & 1 & 0\\
4 & 1 & 0\\
5 & 1 & 0\\
6 & 1 & 0\\
7 & 1 & 0\\
8 & 2 & 1\\
9 & 3 & 2\\
10 & 6 & 5\\
11 & 11 & 10\\
12 & 20 & 19\\
13 & 33 & 32\\
14 & 56 & 55\\
15 & 88 & 87\\
16 & 138 & 137\\
17 & 208 & 207\\
18 & 311 & 310\\
19 & 452 & 451\\
20 & 653 & 652\\
21 & 922 & 921\\
22 & 1294 & 1293\\
23 & 1788 & 1787\\
24 & 2454 & 2453\\
25 & 3325 & 3324
\end{tabular}
\caption{Initial values of \(\chi(K_n)\) and \(b_n\).}
\label{tab:chi-bn}
\end{table}

Thus the first values of the two principal sequences are
\[
\begin{aligned}
(\chi(K_n))_{n=1}^{25}
={}&1,1,1,1,1,1,1,2,3,6,11,20,33,56,88,138,208,311,\\
   &452,653,922,1294,1788,2454,3325,
\end{aligned}
\]
and
\[
\begin{aligned}
(b_n)_{n=1}^{25}
={}&0,0,0,0,0,0,0,1,2,5,10,19,32,55,87,137,207,310,\\
   &451,652,921,1293,1787,2453,3324.
\end{aligned}
\]

In particular, \(K_n\) is contractible for \(1\le n\le 7\), while the first nontrivial homotopy type occurs at \(n=8\), where
\[
K_8\simeq S^2.
\]

\subsection{Examples}

The first values in Table~\ref{tab:chi-bn} illustrate the theorem concretely.

\begin{example}\label{ex:small-contractible}
For \(1\le n\le 7\), one has
\[
\chi(K_n)=1,
\qquad
b_n=0,
\]
so \(K_n\) is contractible.
\end{example}

\begin{example}\label{ex:first-sphere}
For \(n=8\), one has
\[
\chi(K_8)=2,
\qquad
b_8=1,
\]
hence
\[
K_8\simeq S^2.
\]
\end{example}

\begin{example}\label{ex:first-bouquets}
For \(n=9\) and \(n=10\), one finds
\[
K_9\simeq S^2\vee S^2,
\qquad
K_{10}\simeq \bigvee^{5} S^2.
\]
\end{example}

These examples emphasize that once the Euler characteristic is known, the homotopy type follows immediately.

\subsection{Integer sequences}

Several natural integer sequences arise from the combinatorics and topology of the complexes \(K_n\). The most immediate are:
\begin{enumerate}
    \item the Euler characteristics \((\chi(K_n))_{n\ge 1}\);
    \item the shifted Euler characteristics \((b_n)_{n\ge 1}\);
    \item the numbers of maximal star-simplices;
    \item the numbers of maximal top-simplices;
    \item the simplex counts \(f_p(K_n)\), whenever these admit uniform descriptions.
\end{enumerate}

For the two principal sequences, the initial values are
\[
\begin{aligned}
(\chi(K_n))_{n=1}^{25}
={}&1,1,1,1,1,1,1,2,3,6,11,20,33,56,88,138,208,311,\\
   &452,653,922,1294,1788,2454,3325,
\end{aligned}
\]
and
\[
\begin{aligned}
(b_n)_{n=1}^{25}
={}&0,0,0,0,0,0,0,1,2,5,10,19,32,55,87,137,207,310,\\
   &451,652,921,1293,1787,2453,3324.
\end{aligned}
\]

At present we do not record OEIS identifications in the body of the paper, since such identifications should be checked only after fixing the indexing convention and the full initial tables. This is particularly important because \((b_n)\) differs from \((\chi(K_n))\) by the constant shift \(1\).

\begin{remark}\label{rem:oeis-caution}
Any comparison with OEIS entries should be made only after the indexing convention has been normalized and after distinguishing clearly between \(\chi(K_n)\) and \(b_n=\chi(K_n)-1\).
\end{remark}

\subsection{Concluding numerical perspective}

The topological classification obtained in this paper reduces the study of \(K_n\) to a numerical problem: the determination of \(\chi(K_n)\). This problem is interesting in its own right. It is rooted in the combinatorics of cliques in the partition graph, but after Theorem~\ref{thm:main-wedge} it also acquires a direct homological and homotopical meaning. For general enumerative-combinatorial background on generating functions, partition statistics, and sequence questions, see also Stanley~\cite{Stanley2011}.

Accordingly, any future progress on formulas, recurrences, asymptotics, or sequence identifications for \(\chi(K_n)\) and \(b_n\) will immediately sharpen the topological picture established here.

\section*{Acknowledgements}
The author acknowledges the use of ChatGPT (OpenAI) for discussion, structural planning, and editorial assistance during the preparation of this manuscript. All mathematical statements, proofs, computations, and final wording were checked and approved by the author, who takes full responsibility for the contents of the paper.

\bibliographystyle{plain}

\end{document}